\documentclass[12pt]{amsart}
\usepackage{amssymb,amscd}
\usepackage{tikz-cd}
\usepackage{amsmath}
\usepackage{scalerel,stackengine}
\stackMath
\newcommand\reallywidehat[1]{\savestack{\tmpbox}{\stretchto{\scaleto{\scalerel*[\widthof{\ensuremath{#1}}]{\kern-.6pt\bigwedge\kern-.6pt}{\rule[-\textheight/2]{1ex}{\textheight}}}{\textheight}}{0.5ex}}\stackon[1pt]{#1}{\tmpbox}}

\newcommand{\cupdot}{\mathbin{\mathaccent\cdot\cup}}
\begin{document}
\newtheorem{thm}{Theorem}[section]
\newtheorem{lem}[thm]{Lemma}
\newtheorem{prop}[thm]{Proposition}
\newtheorem{cor}[thm]{Corollary}
\theoremstyle{definition}
\newtheorem{ex}[thm]{Example}
\newtheorem{rem}[thm]{Remark}
\newtheorem{prob}[thm]{Problem}
\newtheorem{thmA}{Theorem}
\renewcommand{\thethmA}{}
\newtheorem{defi}[thm]{Definition}
\renewcommand{\thedefi}{}
\long\def\alert#1{\smallskip{\hskip\parindent\vrule%
\vbox{\advance\hsize-2\parindent\hrule\smallskip\parindent.4\parindent%
\narrower\noindent#1\smallskip\hrule}\vrule\hfill}\smallskip}
\def\ff{\frak}
\def\Spec{\mbox{\rm Spec}}
\def\type{\mbox{ type}}
\def\Hom{\mbox{ Hom}}
\def\rank{\mbox{ rank}}
\def\Ext{\mbox{ Ext}}
\def\Ker{\mbox{ Ker}}
\def\Max{\mbox{\rm Max}}
\def\End{\mbox{\rm End}}
\def\l{\langle\:}
\def\r{\:\rangle}
\def\Rad{\mbox{\rm Rad}}
\def\Cont{\mbox{\rm Cont}}
\def\Zar{\mbox{\rm Zar}}
\def\Supp{\mbox{\rm Supp}}
\def\Rep{\mbox{\rm Rep}}
\def\cal{\mathcal}
\title[Profinite completions and MacNeille completions of MV-algebras]{Profinite completions and MacNeille completions of MV-algebras}
\thanks{\today}
\author{Jean B. Nganou}
\address{Department of Mathematics, University of Oregon, Eugene,
OR 97403} \email{nganou@uoregon.edu}
\begin{abstract} We provide a concrete description of the profinite completion of an arbitrary MV-algebra, a description that generalizes the well known profinite completion of a Boolean algebra as the power set of its Stone space. We also use the description found to investigate profinite MV-algebras that are profinite completions of MV-algebras. Finally, we characterize semisimple MV-algebras for which the profinite completion and the MacNeille completion are isomorphic. 
\vspace{.2cm}\\
\end{abstract}
\maketitle
\section{Introduction}
Profinite completions and MacNeille completions have been thoroughly investigated in many important varieties of algebras. Some of the most popular varieties include the variety $\mathbf{DL}$ of bounded distributive lattices, the variety $\mathbf{HA}$ of Heyting algebras, and the variety $\mathbf{BA}$ of Boolean algebras. It is known that the profinite completion of a Boolean algebra is isomorphic to the power set of its Stone space  (\cite[Corollary 3.3]{BM}), that the profinite completion of a bounded distributive lattice with Priestley space $X$ is isomorphic to the lattice Up($X$) of upsets of $X$ (\cite[Theorem 3.1]{BM}), and that the completion of a Heyting algebra with Esakia space $X$ is isomorphic to the Heyting algebra Up($X_{\text{fin}}$) of the upsets of $X_{\text{fin}}$ (\cite[Theorem 4.7]{BGMM}).\par
Besides $\mathbf{DL}$ and $\mathbf{HA}$, the variety $\mathbf{MV}$ of MV-algebras is another popular extension of $\mathbf{BA}$. MV-algebras constitute the algebraic counterpart of the \L ukasiewicz many valued logic and via the Chang-Mundici functor are equivalent to abelian $\ell$-groups with distinguished units.
 More precisely, an MV-algebra can be defined as an Abelian monoid $(A,\oplus , 0)$ with an involution $\neg:A\to A$ (i.e., $\neg \neg x=x$ for all $x\in A$) satisfying the following axioms for all $x, y\in A$: $\neg 0\oplus x=\neg0$,
$\neg(\neg x\oplus y)\oplus y=\neg(\neg y\oplus x)\oplus x$. For any $x,y\in A$, if one writes $x\leq y$ when $\neg x\oplus y=\neg 0:=1$, then $\leq$ induces a partial order on $A$, which is in fact a lattice order where $x\vee y=\neg(\neg x\oplus y)\oplus y$ and $x\wedge y=\neg(\neg x\vee \neg y)$. An ideal of an MV-algebra is a nonempty subset $I$ of $A$ such that (i) for all $x,y\in I$, $x\oplus y\in I$ and (ii) for all $x\in A$ and $y\in I$, $x\leq y$ implies $x\in I$. A prime ideal of $A$ is proper ideal $P$ such that whenever $x\wedge y\in P$ with $x, y\in A$, then $x\in P$ or $y\in P$. A maximal ideal of $A$ is proper ideal $M$ such that for every $a\in A$ such that $a\notin M$, there exists a integer $n\geq 1$ such that $\neg na\in M$.\par
A topological MV-algebra is an MV-algebra $(A,\oplus, \neg, 0)$ together with a topology $\tau$ such that $\oplus$ and $\neg$ (and in particular $\vee$, $\wedge$) are $\tau$-continuous. A detailed treatment of topological MV-algebras can be found in \cite{Hoo, WH} 

Unlike bounded distributive lattices, Heyting algebras, Boolean algebras, or orthomodular lattices, where the topic of profiniteness has been well investigated (see for e.g., \cite{BGMM, BM, BV, CG}), MV-algebras have not yet received the same level of attention. To continue our study on the theme of profiniteness in MV-algebras, which has been initiated in \cite{jbn, jbn2}, we focus in this article on profinite completions and applications.\par
In the first part, we compute the profinite completion of every MV-algebra. We obtain that the profinite completion of an MV-algebra is the direct product of all its finite simple homomorphic images. As immediate byproducts of our description, we obtain simpler proofs of some previously known results such as the profinite completion of a Boolean algebra, the preservation of the Boolean center of regular MV-algebras  by profinite completions, and the characterization of MV-algebras that are isomorphic to their own profinite completions. \par
In the second part of the article, we use the description of the profinite completion found to characterize profinite MV-algebras that are isomorphic to profinite completions of some MV-algebras. Among other things, we prove that a profinite MV-algebra $A:= \prod_{x\in X}\L_{n_x}$ is isomorphic to the profinite completion of an MV-algebra if and only if  there exists a Stone space $Y$ containing a dense copy $X'$ of $X$ and a separating subalgebra $A'$ of $\Cont(Y)$ satisfying:\\
(i) For every $x\in X$ (when $X$ is identified with $X'$), $J_x$ has rank $n_x$ in $A'$, where $J_x:=\{f\in A':f(x)=0\}$; and (ii) For every $y\in Y\setminus X$, $J_y$ has infinite rank in $A'$.

Besides the profinite completion, another popular completion that has been well studied on algebras is the MacNeille completion. The MacNeille completion of MV-algebras has been previously investigated in \cite{BGL, DP, Ja}. In the last final part of the paper, we characterize semisimple MV-algebras for which the profinite completion and the MacNeille completion are isomorphic. We obtain that these are the atomic MV-algebras $A$ for which there exists a bijection $\tau$, from the set $\mathfrak{a}(A)$ of atoms of $A$ onto its set $\text{Max}_{f}(A)$ of maximal ideals of finite rank such that $\rank(\tau(a))=|a|+1$ for all atom $a$ of $A$.

The main results of the paper extend naturally the existing ones from the theory of Boolean algebras (Corollary \ref{Bool}, Theorem \ref{main2}, Corollary \ref{bounded} and Corollary \ref{Bool1}). They also offer simpler proofs to some previously obtained results on MV-algebras (Corollary \ref{self} and Corollary \ref{preserve}). 

For basic MV-algebra terminologies, the reader can consult \cite{C2, Mu}, and for basic facts about profinite MV-algebras, the reader can consult \cite{jbn, jbn2}. 
\section{profinite completions of MV-algebras}
Let $A$ be an MV-algebra and let $\text{id}_f(A)$ be the set of all ideals $I$ of $A$ such that $A/I$ is finite. For every $I, J\in \text{id}_f(A)$ such that $I\subseteq J$, let $\phi_{JI}:A/I\to A/J$ be the natural homomorphism, i.e., $\phi_{JI}([a]_I)=[a]_J$ for all $a\in A$. Observe that $(\text{id}_f(A), \supseteq)$ is a directed set, and $\left\{(\text{id}_f(A), \supseteq), \{A/I\},\{\phi_{JI}\}\right\}$ is an inverse system of MV-algebras. The inverse limit of this inverse system is called the profinite completion of the MV-algebra $A$, and commonly denoted by $\widehat{A}$. The following description of $\widehat{A}$ is also well known:
$$\widehat{A}\cong \left\{\alpha\in \prod_{I\in \text{id}_f(A)}A/I: \phi_{JI}(\alpha(I))=\alpha(J) \; \text{whenever}\; I \subseteq J\right\} $$
In the sequel, the set of prime ideals of $A$ will be denoted by $\text{Spec}(A)$ and is endowed with the Zariski's topology. The set $\text{Max}(A)$ of maximal ideals of $A$ inherits the subspace topology of $\text{Spec}(A)$.\\
Let $A$ be an MV-algebra, and $M$ a maximal ideal of $A$. If $A/M$ is finite, then $A/M\cong \L_n$ for some integer $n\geq 2$, which is called the rank of $M$ and in this case $M$ is said to have finite rank. If $A/M$ is infinite, then $M$ is said to have infinite rank.

Let $\text{Max}_f(A):=\text{Max}(A)\cap \text{id}_f(A)$ is the set of maximal ideals of finite rank in $A$.

The following result, which is a strengthened and more complete version of \cite[Proposition 3.1]{jbn2}, is crucial for our subsequent results.
\begin{prop}\label{frank} Let $A$ be an MV-algebra.
\begin{itemize}
\item[1.] If $I\in \text{id}_f(A)$, then exist $M_1, M_2,\ldots, M_r\in \text{Max}_f(A)$ such that $I=M_1\cap M_2\cap \ldots \cap M_r$ and $A/I\cong \prod_{i=1}^rA/M_i$. Furthermore, the set $\mathcal{S}(I):=\{M_1, M_2,\ldots, M_r\}\subseteq \text{Max}_f(A)$ is uniquely determined by $I$.
\item[2.] For every $I, J\in \text{id}_f(A)$, if $I\subseteq J$, then $\mathcal{S}(J)\subseteq \mathcal{S}(I)$.
\end{itemize}
\end{prop}
\begin{proof}
1. Suppose that $I$ is an ideal of $A$, with $A/I$ finite. Then, by \cite[Proposition 3.6.5]{C2}, there is an isomorphism $\varphi:A/I\to \prod_{i=1}^r\L_{n_i}$, for some integers $n_1, n_2, \ldots, n_r\geq 2$. For each $k=1, 2, \ldots, r$, let $M_k=\ker (q_k\circ \varphi\circ p_I)$, where $q_k$ is the natural projection $\prod_{i=1}^r\L_{n_i}\to \L_k$ and $p_I:A\to A/I$ is the canonical projection. Then $M_k$ is a maximal ideal of $A$ since $A/M_k\cong \L_k$, which is simple. In addition, it is clear that $I=M_1\cap M_2\cap \ldots \cap M_r$ and $A/I\cong \prod_{i=1}^rA/M_i$. Finally, for the uniqueness of the maximal ideals, suppose that $M_1\cap M_2\cap \ldots \cap M_r=M'_1\cap M'_2\cap \ldots \cap M'_s$. Then for every $j=1, \ldots, s$, $M_1\cap M_2\cap \ldots \cap M_r\subseteq M'_j$. But, since each $M_i$ is prime (as any maximal ideal is), then $M_i\subseteq M'_j$ for some $i$. It follows from the maximality of $M_i$ and $M'_j$ that $M_i=M'_j$. Similarly, for every $i=1, \ldots, r$, there exists $j=1, \ldots, s$ such that $M_i=M'_j$. Hence both $M_1, M_2,\ldots, M_r$ and $M'_1, M'_2,\ldots, M'_s$ determine the same set of distinct maximal ideals.\\
2. Suppose that $I=M_1\cap M_2\cap \ldots \cap M_r\subseteq M'_1\cap M'_2\cap \ldots \cap M'_s=J$. Then, as seen in the preceding proof of uniqueness, for every $j=1, \ldots, s$, there exists $i=1, \ldots, r$ such that $M_i=M'_j$. Thus, $\mathcal{S}(J)\subseteq \mathcal{S}(I)$ as claimed.
\end{proof}
\begin{rem}\label{keyisom}
For $I\in \text{id}_f(A)$, consider the homomorphism $\varphi_I:A/I\to \prod_{M\in \mathcal{S}(I)}A/M$ defined by $\varphi_I([a]_I)(M)=[a]_M$ for all $a\in A$ and $M\in \mathcal{S}(I)$. Then, as $I=\cap_{M\in \mathcal{S}(I)}M$, $\varphi_I$ is one-to-one. In addition, by Proposition \ref{frank}(1.) $A/I$ and $\prod_{M\in \mathcal{S}(I)}A/M$ are isomorphic and finite (therefore have same cardinality), thus $\varphi_I$ is indeed an isomorphism.
\end{rem}
We shall use the following notations throughout the paper. We set $\mathcal{I}:=\{ \mathcal{S}(I):I\in \text{id}_f(A)\}$ and for each $I\in \text{id}_f(A)$, $A_I:=\prod_{M\in \mathcal{S}(I)}A/M$. Note that when $\mathcal{S}(J)\subseteq \mathcal{S}(I)$, there is a homomorphism $\mu_{JI}: A_I\to A_J$, namely the natural projection. It is easy to verify that $\{(\mathcal{I}, \subseteq), \{A_I\}, \{\mu_I\}\}$ is an inverse system.
\begin{lem}\label{system}
For every MV-algebra $A$, its profinite completion is isomorphic to the inverse limit of $\{(\mathcal{I}, \subseteq), \{A_I\}, \{\mu_I\}\}$.
\end{lem}
\begin{proof}
It is enough to show that  $\left\{(\text{id}_f(A), \supseteq), \{A/I\},\{\phi_{JI}\}\right\}$ and $\{(\mathcal{I}, \subseteq), \{A_I\}, \{\mu_I\}\}$ are isomorphic inverse systems. To see this, first note that $\mathcal{S}:\text{id}_f(A)\to \mathcal{I}$ is an inclusion reversing bijection. For each $I\in \text{id}_f(A)$, $\varphi_I:A/I\to A_I$ (as defined in Remark \ref{keyisom}) is an isomorphism. To complete the proof, one needs to verify that for all $I \subseteq J$, the following diagram is commutative.
$$
\begin{CD}
A/I @>\phi_{JI}>> A/J\\
@V\varphi_{I}VV @VV\varphi_{J}V\\
A_I @>\mu_{JI}>> A_J
\end{CD}
$$
This commutativity follows easily from the various definitions.
\end{proof}
Next, the main Theorem of this section.
\begin{thm}\label{main}
For every MV-algebra $A$, its profinite completion is algebraically and topologically isomorphic  to $$ \prod_{M\in \text{Max}_{f}(A)} A/M$$
\end{thm}
\begin{proof}
Using Lemma \ref{system}, we shall prove that $\prod_{M\in \text{Max}_{f}(A)} A/M$ is the inverse limit of the inverse system $\{(\mathcal{I}, \subseteq), \{A_I\}, \{\mu_I\}\}$. We achieve this by defining projections $p_I:\prod_{M\in \text{Max}_{f}(A)} A/M\to A_I$ and showing $\left( \prod_{M\in \text{Max}_{f}(A)} A/M, p_I\right)$ has the universal property.\\
1. Let $I\in \text{id}_f(A)$, we have the natural projection $p_I:\prod_{M\in \text{Max}_{f}(A)} A/M\to A_I$. More explicitly, for every $\alpha \in \prod_{M\in \text{Max}_{f}(A)}A/M$, and $M\in \mathcal{S}(I)$, $p_I(\alpha)(M)=\alpha(M)$. It is clear that $\mu_{JI}\circ p_I=p_J$ for all $I, J\in \in \text{id}_f(A)$ such that $\mathcal{S}(J)\subseteq \mathcal{S}(I)$.\\
2. Let $B$ be an MV-algebra together with a family $q_I:B\to A_I$ of homomorphisms that are compatible with the transition morphisms $\{\mu_{JI}\}$. Define $$\Theta: B\to \prod_{M\in \text{Max}_{f}(A)} A/M\; \; \; \text{by}\; \; \; \Theta(b)(M)=q_M(b)$$ for all $b\in B$, and $M\in \text{Max}_{f}(A)$. Note that for every $M\in \text{Max}_{f}(A)$, $A_M=A/M$ since $\mathcal{S}(M)=\{M\}$. It follows that $\Theta$ is a well-defined homomorphism. \par
We need to show that $\Theta$ is the unique homomorphism that makes the following diagram commutative.
\[
\begin{tikzcd}
  {} & B \arrow{d}{q_I} \arrow[dashed,swap]{dl}{ \Theta} \\
 \prod_{M\in \text{Max}_{f}(A)} A/M \arrow{r}{p_I} & A_I
\end{tikzcd}
\]
That is, $p_I\circ \Theta=q_I$ for all $I\in \text{id}_f(A)$.\\
(i) First, it is clear by construction of $\Theta$ that $p_M\circ \Theta=q_M$ for all $M\in \text{Max}_{f}(A)$. \\
(ii) Second, we prove that $q_I(b)(M)=q_M(b)$ and $p_I(\alpha)(M)=p_M(\alpha)$ for all $M\in \mathcal{S}(I)$. For the first equality, observe that $M\in \mathcal{S}(I)$ means $I\subseteq M$, and by the compatibility of the $q_I$'s, we have $\mu_{MI}\circ q_I=q_M$, which is exactly the first equality. The second equality is clear from the definition of $p_I$.\\
Now, let $I\in \text{id}_f(A)$, $b\in B$ and $M\in \mathcal{S}(I)$. Then, it follows from (i) and (ii) that: $(p_I\circ \Theta)(b)(M)=p_I(\Theta(b))(M)=p_M(\Theta(b))=q_M(b)=q_I(b))(M)$. Thus, $p_I\circ \Theta=q_I$ for all $I\in \text{id}_f(A)$ as claimed.\par
It remains to address the uniqueness of $\Theta$. \par 
Suppose that $\Theta': B\to \prod_{M\in \text{Max}_{f}(A)} A/M$ is a homomorphism such that $p_I\circ \Theta'=q_I$ for all $I\in \text{id}_f(A)$. Then for every $b\in B$ and $M\in \mathcal{S}(I)$, $(p_I\circ \Theta')(b)(M)=q_I(b)(M)$, that is $p_I(\Theta'(b))(M)=q_I(b)(M)$, or $\Theta'(b)(M)=q_I(b)(M)$. The last equation is exactly the definition of $\Theta$. Thus, $\Theta'=\Theta$, and the proof is complete.

Finally, we observe that any isomorphism between Stone MV-algebras is automatically a homeomorphism. This can be seen using \cite[Thm. 2.3]{jbn2} \cite[Prop. 3.5]{jbn}.
\end{proof}
As an immediate consequence of this Theorem, we obtain a well-known description of the profinite completion of a Boolean algebra.
\begin{cor}\label{Bool}
For every Boolean algebra $B$, its profinite completion is given by $$\widehat{B}\cong \mathcal{P}(X)$$
where, $X$ is the Stone space of $B$.
\end{cor}
\begin{proof}
Note that since $B$ is a Boolean algebra, then every maximal ideal of $B$ has finite rank ( rank $2$). Indeed, for every maximal ideal $M$ of $B$, $B/M$ is a simple Boolean algebra and the only simple Boolean algebra is the 2-element Boolean algebra $\mathbf{2}$. Therefore, $\text{Max}_{f}(B)=\text{Max}(B)=X$, the Stone space of $B$. By Theorem \ref{main}, we obtain $\widehat{B}\cong \prod_{M\in X}\mathbf{2}\cong \mathbf{2}^X\cong \mathcal{P}(X)$.
\end{proof}
The next two results were established in \cite{jbn2}, but we can now offer simpler proofs thanks to Theorem \ref{main}.
\begin{cor}\cite[Theorem 3.8]{jbn2}\label{self}
An MV-algebra $A$ is isomorphic to its profinite completion if and only if $A$ is profinite and every maximal ideal of $A$ of finite rank is principal.
\end{cor}
\begin{proof}
Let $\text{PMax}(A)$ denotes the set of principal maximal ideals of $A$. We also observe from \cite[Theorem 2.3]{jbn2} that every isomorphic profinite MV-algebras are automatically homeomorphic. \\
$\Rightarrow:)$ Suppose that $A$ is isomorphic to its profinite completion. Then, $A$ is profinite since the profinite completion is profinite by definition. Therefore $A$ and $\widehat{A}$ are both isomorphic and homeomorphic. Let $M$ be a maximal ideal of $A$ with finite rank. Since from Theorem \ref{main}, $\widehat{A}\cong \prod_{M\in \text{Max}_{f}(A)} A/M$, then the projection $P_M:\widehat{A}\to A/M$ is continuous. Thus, the natural projection $A\to A/M$ is continuous.  Whence, $M$ is clopen (see for e.g. \cite[Theorem 3.12]{Hoo}), and by \cite[Lemma 3.1]{jbn}, $M$ is principal.\\
$\Leftarrow:)$ Suppose that $A$ is profinite and every maximal ideal of $A$ of finite rank is principal. Then, since $A$ is profinite, it follows from \cite[Theorem 2.5, Lemma 3.1] {jbn} that $A\cong \prod_{M\in \text{PMax}(A)} A/M$. It now follows from Theorem \ref{main}, and the fact that $\text{Max}_{f}(A)=\text{PMax}(A)$ that, $A\cong \prod_{M\in \text{PMax}(A)} A/M \cong  \prod_{M\in \text{Max}_{f}(A)} A/M\cong \widehat{A}$. Hence, $A\cong \widehat{A}$ as claimed.
\end{proof}
Recall \cite{BNS} that  an MV-algebra $A$ is called regular if for every prime ideal $N$ of its Boolean center $B(A)$, the ideal of $A$ generated by $N$ is a prime ideal of $A$. 
\begin{cor}\label{perserves}
Let $A$ be a regular MV-algebra in which every maximal ideal has finite rank. Then, $$B(\widehat{A})\cong \reallywidehat{B(A)}$$
\end{cor}
\begin{proof}
First we observe as $A$ is regular that $\text{Max}(A)$ is homeomorphic to the Stone space of $B(A)$ \cite[Proposition 26]{BNS}.  

By assumption and Theorem \ref{main}, we have $\widehat{A}\cong \prod_{M\in \text{Max}(A)} A/M$. Hence, $B(\widehat{A})\cong \prod_{M\in \text{Max}(A)} \mathbf{2}\cong \mathbf{2}^{\text{Max}(A)}$. 
In addition, since $A$ is regular, it follows from the homeomorphism stated above that, $\mathbf{2}^{\text{Max}(A)}\cong \mathbf{2}^{X}\cong \mathcal{P}(X)$, where $X$ is the Stone space of $B(A)$. Therefore, by Corollary \ref{Bool}, $\mathcal{P}(X)\cong \reallywidehat{B(A)}$ and $B(\widehat{A})\cong \reallywidehat{B(A)}$ as needed.

\end{proof}

Profinite completions preserve finite products as we prove next.
\begin{prop}\label{preserve}
For every positive integer $n$ and $A_1, A_2,\ldots, A_n$ MV-algebras, $$\reallywidehat{\prod_{i=1}^nA_i}\cong \prod_{i=1}^n\widehat{A_i}$$
\end{prop}
\begin{proof}
It is enough to prove the result for $n=2$. We observe (as in classical ring theory) that maximal ideals of $A_1\times A_2$ are of the form $M_1\times A_2$ with $M_1$ a maximal ideal of $A_1$ or $A_1\times M_2$ with $M_2$ a maximal ideal of $A_2$. In addition since $A_1\times A_2/M_1\times A_2\cong A_1/M_1$ and $A_1\times A_2/A_1\times M_2\cong A_2/M_2$, it follows that: $$\text{Max}_{f}(A_1\times A_2)=\{M_1\times A_2:M_1\in \text{Max}_{f}(A_1)\}\cupdot \{A_1\times M_2:M_2\in \text{Max}_{f}(A_2)\}$$
From Theorem \ref{main}, we obtain that:
\[
\begin{aligned}
\reallywidehat{A_1\times A_2}&\cong \prod_{M\in \text{Max}_{f}(A_1\times A_2)}(A_1\times A_2)/M\\
&\cong \prod_{M_1\in \text{Max}_{f}(A_1)}(A_1\times A_2)/(M_1\times A_2)\times \prod_{M_2\in \text{Max}_{f}(A_2)}(A_1\times A_2)/(A_1\times M_2)\\
&\cong  \prod_{M_1\in \text{Max}_{f}(A_1)}A_1/M_1\times \prod_{M_2\in \text{Max}_{f}(A_2)}A_2/M_2\\
&\cong \widehat{A_1}\times \widehat{A_2}.
\end{aligned}
\]
\end{proof}
We would like to point out that the above does not extend to arbitrary products. Indeed, if the profinite completions preserves infinite product, then for every infinite set $X$, the Boolean algebra $\mathbf{2}^X$ would be isomorphic to its own profinite completion. However no infinite Boolean algebra is isomorphic to its own profinite completion.

By the Stone's most celebrated result, the Stone space completely determines the Boolean algebra. It follows from Corollary \ref{Bool} that two nonisomorphic Bolean algebras cannot have isomorphic profinite completions. This is however not the case for MV-algebras in general. For instance, nontrivial MV-algebras can have trivial profinite completions. Note that it follows from Theorem \ref{main} that an MV-algebra has a trivial profinite completion if and only if all its maximal ideals have infinite ranks. Infinite simple MV-algebras are the most evident examples of such MV-algebras. 
\section{Profinite MV-algebras vs Profinite completions vs MacNeille completions}
While any MV-algebra $A$ with non-principal maximal ideals of finite ranks is not isomorphic to its profinite completion $\widehat{A}$ (Proposition \ref{self}), it remains possible to have $A\cong \widehat{B}$ for some MV-algebra $B$. This motivates the important question of which (profinite) MV-algebras are profinite completions of some MV-algebras. It is known \cite[Corollary 5.10]{BM} that every profinite Boolean algebra is a profinite completion of some Boolean algebra. It is also known that a profinite bounded distributive lattice is isomorphic to a profinite completion if and only if it is isomorphic to the lattice of upsets of a representable poset \cite[Theorem 5.3]{BM}. 

In this section, we investigate this question for MV-algebras.

Recall that the radical $\Rad(A)$ of an MV-algebra $A$ is the intersection of all its maximal ideals. An MV-algebra $A$ is called semisimple if $\Rad(A)=\{0\}$.
\begin{prop}\label{semi}
For every MV-algebra $A$, $A/\Rad(A)$ and $A$ have isomorphic profinite completions, that is:
$$\reallywidehat{A/\Rad(A)}\cong \widehat{A}$$
\end{prop}
\begin{proof}
We recall that maximal ideals of $A/\Rad(A)$ are of the form $M/\Rad(A)$, where $M$ is a maximal ideal of $A$ containing $\Rad(A)$.  Since $\Rad(A)$ is the intersection of all maximal ideals of $A$, then maximal ideals of $A/\Rad(A)$ are exactly of the form $M/\Rad(A)$, where $M$ is a maximal ideal of $A$. In addition since $(A/\Rad(A))/(M/\Rad(A))\cong A/M$, $M$ has finite rank in $A$ if and only if $M/\Rad(A)$ has finite rank (as both ranks are equal). Thus by Theorem \ref{main}, 
\[ 
\reallywidehat{A/\Rad(A)}\cong \prod_{M\in \text{Max}_{f}(A)} (A/\Rad(A))/(M/\Rad(A))\cong \prod_{M\in \text{Max}_{f}(A)} A/M\cong \widehat{A}.
\]
\end{proof}

Note that given any nonempty topological space $X$, the set $\Cont(X)$ of all continuous functions from $X\to [0,1]$ is an MV-algebra under pointwise operations. Recall that a subalgebra $B$ of $\Cont(X)$ is called separating if for every $x, x'\in X$ with $x\ne x'$, there exists $f\in B$ such that $f(x)\ne f(x')$.

Recall that if $A$ is a profinite MV-algebra, then $A\cong \prod_{x\in X}\L_{n_x}$ for some set $X$, and the set of integers $\eta(A):=\{n_x:x\in X\}$ is uniquely determined by $A$.  
\begin{prop}\label{separate}
Let $A:= \prod_{x\in X}\L_{n_x}$ be a profinite MV-algebra.\\ Then the following assertions are equivalent.
\begin{itemize}
\item[1.] $A$ is isomorphic to the profinite completion of some MV-algebra $B$;
\item[2.] There exists a compact Hausdorff space $Y$ containing a dense copy $X'$ of $X$ and a separating subalgebra $A'$ of $\Cont(Y)$ satisfying:
\begin{itemize}
\item[(i)] For every $x\in X$ (when $X$ is identified with $X'$), $J_x$ has rank $n_x$ in $A'$, where $J_x:=\{f\in A':f(x)=0\}$; and
\item[(ii)] For every $y\in Y\setminus X$, $J_y$ has infinite rank in $A'$;
\end{itemize}
\item[3.] The condition 2 holds for some Stone space $Y$.
\item[4.] $A$ is isomorphic to the profinite completion of some sub-MV-algebra of $A$.
\end{itemize}
\end{prop}
\begin{proof}
$1.\Rightarrow 2$.: Suppose that $A$ is isomorphic to the profinite completion of some MV-algebra $B$. Then by Proposition \ref{semi}, $A$ is isomorphic to the profinite completion of $A':=B/\Rad(B)$. But $A'$ is semisimple, therefore by \cite[Cor. 3.6.8]{C2}, $A'$ is isomorphic to a separating MV-algebra of $[0,1]$-valued continuous functions on some nonempty compact Hausdorff space $Y$, with pointwise operations. By \cite[Thm. 3.4.3]{C2}, the maximal ideals of $A'$ are exactly $J_y$, $y\in Y$. Let $X':=\{y\in Y: J_y \; \text{has finite rank in}\; A'\}$. Then by Theorem \ref{main}, $A\cong \widehat{A'}\cong \prod_{y\in X'}A'/J_y$. It follows that there is a bijection $\tau:X\to X'$ such that $n_x=rank(J_{\tau(x)})$ for all $x\in X$. Clearly, for all $y\in Y\setminus X$, $J_y$ has infinite rank in $A'$. Note that the density clause can be added by simply replacing $Y$ by the closure of $X'$ if needed.\\
$2. \Leftrightarrow 3$.: This follows form the fact that every semi-simple MV-algebra is isomorphic to a separating subalgebra $A'$ of $\Cont(Y)$, where $Y$ is a Stone space as proved in \cite[Cor. 3.7]{Glus}.\\
$2. \Rightarrow 4$.: Suppose that $A'$ is an MV-algebra satisfying the condition 2. Since $A'/J_x\cong \L_{n_x}$ for all $x\in X$, then for every $f\in A'$, $f(x)\in \L_{n_x}$ for all $x\in X$. On the other hand, by density, the map $f\mapsto f_{|X} $ is one-to-one from $\Cont(Y)$ into $\Cont(X)$, which is in turn a sub-MV-algebra of $[0,1]^X$. This allows us to identify $\Cont(Y)$ to a sub-MV-algebra of $[0,1]^X$. Hence, $A'$ is a sub-MV-algebra of $\{f\in [0,1]^X:f(x)\in \L_{n_x} \; \text{for all} \; x\in X\}=A$. Finally, note that  by Theorem \ref{main}  $\widehat{A'}\cong  \prod_{y\in X'}A'/J_y\cong \prod_{x\in X}\L_{n_x}=A$.\\
$4. \Rightarrow 1$.: Is clear.
\end{proof}
\begin{thm}\label{main2}
Suppose that $A$ is a profinite MV-algebras such that there exists $n_0\in \eta(A)$ with the property that $n_0-1$ divides all but finitely many $n\in \eta(A)$. Then there exists an MV-algebra $B$ such that $A\cong \widehat{B}$
\end{thm}
\begin{proof}
If $X$ is finite, so is $A$ and in this case $A$ is isomorphic to its profinite completion. We may assume that $X$ is infinite.
Let $x_0\in X$ such that $n_{x_0}\in \eta(A)$ with the property that $n_{x_0}-1$ divides all but finitely many $n\in \eta(A)$. Then $\L_{n_{x_0}}$ is a subalgebra of all but finitely many $\L_{n_x}$, $x\in X$. We topologize $X$ so that the space obtained is the one-point compactification of the discrete space $X\setminus \{x_0\}$. Note that $X$ is infinite, therefore under the discrete topology $X\setminus \{x_0\}$ is a locally compact and Haussdorff space, which is not compact. More explicitly, a subset $U$ of $X$ is open if $U$ or $X\setminus U$ is a finite subset of $X\setminus \{x_0\}$. Then $X$ is a compact Haussdorff (in fact Stone) space. Let  \[ B=\{f\in \Cont(X):f(x)\in \L_{n_x} \; \text{for all} \; x\in X\}\]
Then as $X$ is Stone, the Boolean center $B(\Cont(X))$ is a separating subalgebra of $\Cont(X)$. It follows that $B$ is also a separating subalgebra of $\Cont(X)$ since $B(\Cont(X))=\{f\in \Cont(X):f(x)\in \{0,1\} \; \text{for all} \; x\in X\}=B(B)\subseteq B$. We claim that $B/J_x\cong \L_{n_x}$. To see this, for each $x\in X$, we consider $\varphi_x:B\to \L_{n_x}$ defined by $\varphi_x(f)=f(x)$. Then, $\varphi_x$ is an MV-algebra homomorphism and $\Ker(\varphi_x)=J_x$. It remains to show that each $\varphi_x$ is onto. First we consider the case $x\ne x_0$. For $t\in \L_{n_x}$, define $f:X\to [0,1]$ by:
\[
 \begin{aligned}
f(a)=
\left\{\begin{array}{ll}
 t & \ \ \mbox{if} \ \ a=x;\\
 0&   \ \  \mbox{otherwise}.\\
\end{array}\right.\\
 \end{aligned}
\]
It is easy to see that $f$ is continuous and clearly $\varphi_x(f)=t$.

For $x_0$, recall that $\L_{n_{x_0}}\subseteq \L_{n_x}$ for all $x\in X$, except possibly for $x\in S$, where $S$ is a finite subset of $X$. Now, let $t\in \L_{n_{x_0}}$ and define $f:X\to [0,1]$ by:
\[
 \begin{aligned}
f(a)=
\left\{\begin{array}{ll}
 t & \ \ \mbox{if} \ \ a\notin S;\\
 0&   \ \  \mbox{if} \ \ a\in S.\\
\end{array}\right.\\
 \end{aligned}
\]
It is easy to see that $f$ is continuous and clearly $\varphi_{x_0}(f)=t$.

The conclusion now follows from Proposition \ref{separate}.
\end{proof}
One should observe from Theorem \ref{main2} that if the 2-element Boolean algebra is among the factors of a profinite MV-algebra $A$, then there exists an MV-algebra $B$ such that $A\cong \widehat{B}$. In particular, every profinite Boolean algebra is the profinite completion of some Boolean algebra, which is well known (see for e.g.,\cite[Corollary 5.10]{BM}).
\begin{prop}\label{bounded}
Suppose that $A$ is a profinite MV-algebra with $\eta(A)$ bounded. Then $A$ is isomorphic to the profinite completion of some MV-algebra.
\end{prop}
\begin{proof}
Since $\eta(A)$ is finite, let $\eta(A):=\{n_1,n_2,\ldots, n_k\}$. For each $i=1,2,\ldots, k$, let $X_i=\{x\in X:n_x=n_i\}$ and $A_i=\L_{n_i}^{X_i}$. Then, $A\cong \prod_{i=1}^kA_i$, each $A_i$ clearly satisfies the condition of Theorem \ref{main2}. Therefore, by Theorem \ref{main2} each $A_i$ is isomorphic to the profinite completion of some MV-algebra $B_i$. It follows from Proposition \ref{preserve} that $A\cong \reallywidehat{\prod_{i=1}^nB_i}$.
\end{proof}
While Theorem \ref{main2} and Proposition \ref{bounded} cover a large class of profinite MV-algebras, it remains unclear whether there are profinite MV-algebras that do not satisfy the conditions of Proposition \ref{separate}. This situation is parallel to that of Heyting algebras as described in the comments after the proof of \cite[Corollary 5.9]{BM}.

The final aspect of this article is devoted to comparing the profinite and the MacNeille completions of MV-algebras. Recall that given a lattice $L$, its MacNeille completion is the unique (up to isomorphism) complete lattice $\overline{L}$ with a lattice embedding $\beta:L\to \overline{L}$ such that $\beta[L]$ is both join-dense and meet-dense in $\overline{L}$. More precisely, we wish to investigate under what circumstances the two completions coincide.  For instance, it is known that for a Heyting algebra $H$, the two completions coincide if and only if $H$ is completely join-prime generated and $\langle J^{\infty}(H), \geq \rangle$ is order-isomorphic to the poset of prime filters of $H$ of finite index \cite[Thm. 4.4]{BV}. It is also known that if $L$ is a bounded distributive lattice with Priestley space $X$, then the two completions are isomorphic if and only if $X_0$ is dense in $X$ and $X_0$ is order-isomorphic to $X$, where $X_0$ denotes the subset of order-isolated points in $X$ \cite[Thm. 4.2]{BM}.

For MV-algebras, we start our analysis by observing that the MacNeille completion of an MV-algebra is an MV-algebra if and only if it is semisimple \cite[Thm. 6.3]{Ho}. Therefore, our comparison of the two completions is restricted to semisimple MV-algebras. We shall denote the MacNeille completion of an MV-algebra $A$ by $\overline{A}$. 

Recall that one can define a partial addition $+$ on any MV-algebra $\langle A,\oplus ,\neg, 0\rangle$ by $x+y=x\oplus y$ for all $x, y\in A$ with $x\leq \neg y$. For every $n\geq 1$ integer and $a\in A$, if $a+\cdots +a$ ($n$ times) is defined, we denote it by $nx$. For every $a\in A$, the order of $a$ is defined as $|a|:=\text{Sup}\{n\geq 1: na \; \text{is defined in}\; A\}$. An MV-algebra $A$ is Archimedean if for every $a\in A\setminus \{0\}$, $|a|<\infty$. In addition, Archimedean MV-algebras and semisimple MV-algebras coincide (see for e.g., \cite{C2}).

We have the following characterization of semisimple MV-algebras for which the profinite and MacNeille completions coincide.
\begin{prop}\label{Mac}
For every semisimple MV-algebra $A$, its profinite completion $\widehat{A}$ is isomorphic to its MacNeille completion $\overline{A}$ if and only if $A$ is atomic and there exists a bijection $\tau$, from the set $\mathfrak{a}(A)$ of atoms of $A$ onto its set $\text{Max}_{f}(A)$ of maximal ideals of finite rank such that $\rank(\tau(a))=|a|+1$ for all atoms $a$ of $A$.
\end{prop}
\begin{proof} First, we observe as in the proof of \cite[Thm 6.4.20]{DP} that $A$ and $\overline{A}$ have exactly the same atoms, and that $A$ is atomic if and only if $\overline{A}$ is atomic. 

Assume that $\widehat{A}\cong \overline{A}$, then since $\widehat{A}$ is atomic, then $\overline{A}$ is atomic. It follows from the comments above that $A$ is atomic. Moreover, by \cite[Thm. 6.4.20]{DP}, $\overline{A}\cong \prod_{a\in \mathfrak{a}(A)}\L_{|a|+1}$. Thus, by Theorem \ref{main} $ \prod_{M\in \text{Max}_{f}(A)} A/M\cong \prod_{a\in \mathfrak{a}(A)}\L_{|a|+1}$, and it follows from \cite[Cor. 3.4]{jbn} that there is a bijection between $\mathfrak{a}(A)$ and $\text{Max}_{f}(A)$ with the prescribed requirement. 

Conversely, assume that $A$ is atomic and there exists a bijection $\tau:\mathfrak{a}(A)\to \text{Max}_{f}(A)$ such that $\rank(\tau(a))=|a|+1$ for all $a$. Then, $$\prod_{a\in \mathfrak{a}(A)}\L_{|a|+1}\cong  \prod_{M\in \text{Max}_{f}(A)} A/M$$ and it follows again from Theorem \ref{main} and \cite[Thm. 6.4.20]{DP} that  $\widehat{A}\cong \overline{A}$ as needed.
\end{proof}
Note that every Boolean algebra is a semisimple MV-algebra, all maximal ideals of whom have rank $2$ and all nonzero elements have order $1$. Applying Proposition \ref{Mac} to Boolean algebras and the fact that the profinite completion and the canonical completion of Boolean algebras are isomorphic (see for e.g., \cite[Thm. 2.11]{BGMM}, we obtain the following well known result.
\begin{cor}\cite[Cor. 4.7]{BV}\label{Bool1}
The profinite completion and the MacNeille completion of a Boolean algebra $B$ are isomorphic if and only if $B$ is atomic and its set of atoms and that of ultrafilters have the same cardinality.
\end{cor}
We close with two examples, the first dealing with complete MV-algebras while the second deals with non complete MV-algebras.
\begin{ex}
Let $A=\prod_{x\in X}\L_{n_x}$. Then since $A$ is complete, $A$ is isomorphic to its MacNeille completion. Therefore $\widehat{A}\cong \overline{A}$ if and only if $\widehat{A}\cong A$, if and only if all maximal ideals of finite rank in $A$ are principal.
\end{ex}
\begin{ex}\label{noncomplete}
Let $B=\{f\in \prod_{n=1}^\infty \L_{n+1}: f \; \text{is convergent}\}$. Using the continuity of the operations $\oplus$ and $\neg$ of $[0,1]$, one can deduce that for every convergent sequences $f, g:\mathbb{N}\to [0,1]$, both $\neg f$ and $f\oplus g$ are convergent and $\lim \neg f=1-\lim f$, $\lim (f\oplus g)=\lim f\oplus \lim g$. It follows that $B$ is a sub-MV-algebra of $\prod_{n=1}^\infty \L_{n+1}$, and is therefore an MV-algebra on its own. 

We claim that $\widehat{B}\cong \overline{B}\cong \prod_{n=1}^\infty\L_{n+1}$.

We start by representing $B$ as a separating algebra of an MV-algebra of continuous functions on a Stone space. This makes it easy to determine all the maximal ideals of finite rank in $B$. Consider $X:=\mathbb{N}\cup \{\infty\}$ the one-point compactification of the discrete space $\mathbb{N}$. Let $B'=\{g\in \Cont(X): g(n)\in \L_{n+1}, \; \text{for all}\; n\geq 1\; \text{and}\}$. Consider the map $\Theta: B\to B'$ defined by $\Theta(f)=\tilde{f}$, where $\tilde{f}(n)=f(n)$ for all $n\geq 1$, and $\tilde{f}(\infty)=\lim f$. It is readily verified that $\Theta$ is an MV-algebras isomorphism and $B\cong B'$. In addition, $B'$ is a separating subalgebra of $\Cont(X)$, and it follows that its maximal ideals are of the form $J_{x}:=\{g\in \Cont(X): g(x)=0\}$, with $x\in X$ (\cite[Thm. 3.4.3]{C2}). Now, a similar justification to what was done in the proof of Theorem \label{main2} shows that $J_n$ has rank $n+1$ for all $n\geq 1$. It remains to show that $J_{\infty}$ has infinite rank. To see this, we first observe that for every $n\geq 1$, every closed interval centered at a point of $[0,1]$ of length $\frac{1}{n}$ contains a point of $\L_{n+1}$. It follows that for every $r\in [0,1]$ and $n\geq 1$, then there exists $x_n\in \L_{n+1}$ such that $|x_n-r|\leq \frac{1}{2n}$. Consequently, if one defines $f\in B$ by $f(n)=x_n$, then $\lim f=r$ and $f$ extends uniquely to $\tilde{f}\in B'$ such that $\tilde{f}(\infty)=r$. Therefore, $B'/J_{\infty}\cong [0,1]$ and $J_{\infty}$ has infinite rank as claimed. We deduce from Theorem \ref{main} that $\widehat{B}\cong \prod_{n=1}^\infty\L_{n+1}$.

Moreover, for every $n\geq 1$, let $f_n\in B$ defined by $f_n(n)=\frac{1}{n}$ and $f_n(k)=0$ if $k\ne n$. Then it is immediate that $B$ is atomic, the set of atoms of $B$ is $\{f_n: n\in\mathbb{N}\}$ and by the preceding discussion $\text{Max}_{f}(A)=\{J_n:n\in\mathbb{N}\}$. Therefore, the $B$ satisfies the the conditions stated in Proposition \ref{Mac}. Thus $\widehat{B}\cong \overline{B}$ and
$\widehat{B}\cong \overline{B}\cong \prod_{n=1}^\infty\L_{n+1}$..
\end{ex}
\begin{rem}
As announced, the MV-algebra $B$ of Example \ref{noncomplete} is not complete. Indeed, define $f_{2k}=0$ for all $k$, $f_{2k+1}(k)=1$ and $f_{2k+1}(n)=0$ for all $n\ne k$. Then $\{f_n: n\in\mathbb{N}\}$ is a subset of $B$ that does not have a supremum in $B$.
\end{rem}
\section{conclusion and final remarks}
A description of the profinite completion of any MV-algebra is obtained using its maximal ideals of finite ranks (Theorem \ref{main}). The stated description is used to study profinite MV-algebras that are isomorphic to profinite completions of MV-algebras. Finally, a necessary and sufficient condition for the profinite completion and MacNeille completion to be isomorphic is established (Proposition \ref{Mac}). In the proof of Proposition \ref{Mac}, a description of the MacNeille completion of semisimple atomic MV-algebra was used. In a future project, we would like to find a description of the MacNeille completion of general semisimple MV-algebras which should coincide with the stated description for semisimple MV-algebras. Finally, in addition to profinite and MacNeille completions, canonical extensions have also been introduced and studied. These have been extensively studied on Boolean algebras $\mathbf{BA}$, distributive lattices $\mathbf{DL}$, and Heyting algebras $\mathbf{HA}$ (see for e.g., \cite{BGMM, GHV, JH}). In the same spirit, we would also like to investigate canonical extensions of MV-algebras and compare it to the two completions treated in this paper.
\vspace{1cm}\\
\textbf{Acknowledgements:} I would like to thank Professor G. Bezhanishvili for suggesting the topics of \cite{jbn2} and this article, and more importantly for his very insightful comments during the preparation of these articles.

\end{document}